\documentclass[12pt]{amsart}
\usepackage{graphicx} % Required for inserting images

\usepackage{soul} % \st for striking out
\usepackage{amscd,amssymb,mathtools,aliascnt}
\usepackage[dvipsnames]{xcolor}
\usepackage{graphicx}
\usepackage{calrsfs}
\usepackage[labelfont=rm]{subcaption}
\usepackage{tikz-cd} 
\usepackage{tikz}
\usetikzlibrary{decorations.markings}
\usetikzlibrary{decorations.pathreplacing}
\usetikzlibrary{shapes}
\usetikzlibrary{backgrounds}
\usetikzlibrary{calc}
\usetikzlibrary{shapes.misc, positioning}
\usepackage{mdwlist}
\usepackage{xfrac}
\DeclareCaptionSubType*{figure}

\usepackage[arrow,matrix,graph,frame,poly,arc,tips]{xy}

\usepackage{multirow}
\usepackage{enumerate}
\usepackage{verbatim}
\usepackage{comment}
\usepackage{todonotes}

\DeclarePairedDelimiter{\form}{\langle}{\rangle}

    \DeclarePairedDelimiter{\abss}{\lVert}{\rVert}

    \newcommand\bE{\mathbb{E}}
    
    \newcommand\N{\mathbb{N}}
    \newcommand\bR{\mathbb{R}}

    \newcommand\Z{\mathbb{Z}}

    \newcommand\CC{\mathcal{C}}

    \newcommand\II{\mathcal{I}}

    \newcommand\HH{\mathcal{H}}
    \newcommand\PP{\mathcal{P}}

    \newcommand\Sym{\operatorname{Sym}}

    \newcommand\Id{\operatorname{Id}}

    \newcommand\sse{\subseteq}

    \def\thetitle{A note on the proof of nonamenability of $F$}
    \def\theauthors{{Sang-hyun Kim, Yash Lodha}}
    \usepackage{hyperref}
    \hypersetup{
      colorlinks=false,
      plainpages,
      urlcolor=black,
      linkcolor=black,
      pdftitle=   \thetitle,
      pdfauthor=  {\theauthors}
    }
    \usepackage[capitalise]{cleveref}

    \theoremstyle{plain}
    
    \newtheorem{thm}{Theorem}[section]
    \crefname{thm}{Theorem}{Theorems}
    \newaliascnt{lem}{thm}
    \newtheorem{lem}[lem]{Lemma}
    \aliascntresetthe{lem}
    \crefname{lem}{Lemma}{Lemmas}
    \newaliascnt{cor}{thm}
    
    \aliascntresetthe{cor}
    \crefname{cor}{Corollary}{Corollaries}
    \newaliascnt{prop}{thm}
    \newtheorem{prop}[prop]{Proposition}
    \aliascntresetthe{prop}
    \crefname{prop}{Proposition}{Propositions}
    \newaliascnt{con}{thm}
    
    \aliascntresetthe{con}
    \crefname{con}{Conjecture}{Conjectures}
    \newaliascnt{que}{thm}
    
    \aliascntresetthe{que}
    \crefname{que}{Question}{Questions}
    \newaliascnt{obs}{thm}
    
    \aliascntresetthe{obs}
    \crefname{obs}{Observation}{Observations}
    \newaliascnt{hyp}{thm}
    
    \aliascntresetthe{hyp}
    \crefname{hyp}{Hypothesis}{Hypotheses}
    \newtheorem*{claim*}{Claim}
    
    \newtheorem{claim}{Claim}
    \crefname{claim}{Claim}{Claim}

    \theoremstyle{remark}
    \newaliascnt{exmp}{thm}
    
    \aliascntresetthe{exmp}
    \crefname{exmp}{Example}{Examples}
    \newaliascnt{rem}{thm}
    
    \aliascntresetthe{rem}
    \crefname{rem}{Remark}{Remarks}

    \theoremstyle{definition}
    \newaliascnt{defn}{thm}
    
    \aliascntresetthe{defn}
    \crefname{defn}{Definition}{Definitions}
    
    \newaliascnt{setting}{thm}
    
    \aliascntresetthe{setting}
    \crefname{setting}{Setting}{Settings}

\begin{document}
    \title\thetitle
    \date{\today}
%\keywords{Unit fraction, Ahmes series, Double exponential growth}
%\subjclass[2020]{Primary: 11J72; Secondary: 11D68, 40A05}

     %\keywords{ }
    %\subjclass[2020]{Primary:  20A15, 57S05; Secondary: 03C07, 57S25, 57R50, 57R55,   57R60}

    \author[S. Kim]{Sang-hyun Kim}
    \address{School of Mathematics, Korea Institute for Advanced Study (KIAS), Seoul, 02455, Korea}
    \email{skim.math@gmail.com}
    \urladdr{https://kimsh.kr}

  \author[Y. Lodha]{Yash Lodha}
    \address{Department of Mathematics, Purdue University, West Lafayatte,IN, USA}
    \email{ylodha@purdue.edu}
    \urladdr{https://yl7639.wixsite.com/website}
    
\begin{abstract}
The goal of this note is to provide a human written exposition of OpenAI's recent proof that Thompson's group $F$ is nonamenable.
\end{abstract}
\maketitle

    %%%%%%%%%%%%%%%%%%%%%%%%%%%
    % START of body
    %%%%%%%%%%%%%%%%%%%%%%%%%%%
Recall a group is \emph{amenable} if it admits an \emph{invariant mean} (that is, left-invariant finitely-additive probability measure) on all of its subsets. 
The group $\textup{PL}^+[0,1]$ is the group of orientation preserving, piecewise linear homeomorphisms of $[0,1]$ (with finitely many breakpoints allowed for each element).
Thompson's group $F$ is defined as a subgroup of this group consisting of elements $f$ whose breakpoints lie in dyadic rationals $\mathbb{Z}[\frac{1}{2}]$ and whose slopes, whenever they exist, lie in $2^{\mathbb{Z}}$.
It has been a long standing open problem whether $F$ is amenable. 
The problem was first considered by Richard Thompson \cite{Thompson1973}, popularized by Ross Geoghegan as a central part of his conjecture \cite{CannonFloydParry1996, Zaremsky2016, Lodha2020}, and to our knowledge \cite{GerstenStallings1987} contains the first written appearance.

Recently, OpenAI announced a solution. Using the same idea, we give a streamlined proof, and describe a generalization in Theorem \ref{thm:main2}. 
All the arguments here are based on simplifications of those in the OpenAI annoucement, and no novelty is claimed here. 
No part of this writing, including the references, was produced by AI.

%\sam{cited open ai}
\begin{thm}[\cite{oai-thompson}]\label{thm:main}
Thompson's group $F$ is nonamenable.
\end{thm}

\section{Preliminaries}
%\sam{authors in $=>$ authors of}
%We will need a few elementary facts about the group $F$, as well as the notion of extensive amenability due to the authors of~\cite{JM2013}.
An interval of the form $[\frac{k}{2^n}, \frac{(k+1)}{2^n}]\subset (0,1)$ is called an \emph{internal regular dyadic interval}.
%An \emph{internal} regular dyadic interval is one which is a subset of $(0,1)$. 
We denote the set of such intervals by $\II$.
We fix $D=\Z[\frac{1}{2}]\cap [0,1]$. The following is a well-known and elementary fact.

%\sam{Changed (1) to (i): I usually use (1), (2) for independent statements, while (i) and (ii) for conditions of a single statement}
\begin{lem}\label{lem:transitive}
%For each $n\in \N$, the group $F$ acts transitively on the set $$\{(q_1,\ldots, q_n)\mid q_1<\ldots<q_n, \text{ where }q_i\in D\cap (0,1)\}$$
Given two pairs $(I_0,J_0), (I_1,J_1)\in \II^2$ such that $\sup(I_i)<\inf(J_i)$,  there exists an $h\in F$ such that: 
\begin{enumerate}[(i)]
\item $h(I_0)=I_1$ and $h(J_0)=J_1$.
\item The restrictions of $h$ to $I_0,J_0$ are affine.
\end{enumerate}
%the following hold:
%\begin{enumerate}
%\item For each pair $I,J\in \II$, there is an element $g\in F$ such that $g(I)=J$.
%\item Given two pairs $(I_0,J_0), (I_1,J_1)$ such that $I_i,J_i\in \II$ and $\sup(I_i)<\inf(J_i)$,  there exists an $h\in F$ such that $$h(I_0)=I_1\qquad h(J_0)=J_1$$
%and so that the restrictions of $h$ to $I_0,J_0$ are affine.
%\end{enumerate}
\end{lem}

Consider a set $X$ endowed with a group action of a group $G$.
An \emph{invariant mean} on a set $X$ is a finitely additive probability measure $m$ on the set of all subsets of $X$ such that $m(g\cdot A)=m(A)$.
The action of $G$ on $X$ is called \emph{extensively amenable} if it admits an invariant mean on the set
$$\PP_{\mathrm{fin}}(X):=\{S\sse X\;:\; S\text{ is finite }\}$$
such that for all $M\in\PP_{\mathrm{fin}}(X)$, the set
%\sam{removed $M^*:=$. The notation $\Omega_M$ seems more natural, to be introduced later}
\begin{equation}\label{eq:cone}
\{S\in\PP_{\mathrm{fin}}(X)\;:\; S\supseteq M\}
\end{equation}
has full measure.

The set $\PP_{\mathrm{fin}}(X)$ is naturally endowed with a group structure of an abelian group, where the operation is symmetric difference of sets and each element is an involution. Therefore, $\PP_{\mathrm{fin}}(X)$ is naturally isomorphic to $\oplus_{n\in X}\Z/2\Z$.
One may form the natural group extension given by the lamplighter restricted wreath product $$\PP_{\mathrm{fin}}(X)\rtimes G\cong (\oplus_{n\in X}\Z/2\Z)\rtimes G$$
It was proved in~\cite{JM2013} that the $G$-action on $X$ is extensively amenable if and only if the action of $\PP_{\mathrm{fin}}(X)\rtimes G$ on $\PP_{\mathrm{fin}}(X)$ given by $$(A,g)\cdot P=A\triangle (g\cdot P)$$ is amenable.
Since extensions of amenable groups by amenable groups are amenable, we obtain the following.
%\sam{added [\cite{JMBMS2018}] as they stated this explicitly}
\begin{lem}[\cite{JMBMS2018}]\label{lem:ext}
Every action of an amenable group is extensively amenable.
\end{lem}

\section{Proof of the theorem}
We will establish that the action of $F$ on $D$ is not extensively amenable. Once this is established, \cref{lem:ext} would imply that $F$ is not amenable. 
%\sam{new remark. delete this if it's rather boring}
%\begin{rem}
%Our proof also implies that the action on $[0,1]$ is not extensively amenable. We chose to describe the proof for the countable set $D$ for simplicity.
%\begin{rem}Given two pairs $(I_0,J_0), (I_1,J_1)$ such that $I_i,J_i\in \II$ and $\sup(I_i)<\inf(J_i)$,  there exists an $h\in F$ such that $$h(I_0)=I_1\qquad h(J_0)=J_1$$
%and so that the restrictions of $h$ to $I_0,J_0$ are affine.
%\begin{enumerate}
%\item
%Our proof also implies that the action on $[0,1]$ is not extensively amenable. We chose to describe the proof for the countable set $D$ for simplicity. 
%    \item Note the stark contrast between the amenability and the extensive amenability of actions: the action of $F$ on $X$ is amenable since it preserves the Direc delta measure on $\{0\}$.
%\end{enumerate}
%\end{rem}

Let us now assume by way of contradiction that the action 
of $F$ on $D$ is 
%on $\Omega:=\PP_{\mathrm{fin}}(D)$ is 
extensively amenable. 
So there exists a finitely additive $F$-invariant probability measure $\mu$ on $\Omega:=\PP_{\mathrm{fin}}(D)$ satisfying:
$$
\mu\bigl(\Omega_M)=1
\qquad
\text{for every }M\in \Omega,
$$
where $$\Omega_M:=\{P\in \Omega \mid M\subseteq P\}$$
We fix this measure for the rest of the article.

Let us fix two nonempty $I_0,I_1\in \II$ so that $\sup(I_0)<\inf(I_1)$,
and denote the unique orientation-preserving affine bijections: $$s_j:[0,1]\to I_j\qquad j\in\{0,1\}$$
%$\Omega=\mathbb{Z}[1/2]\cap [0,1]$ and  $$I_0=[1/8,1/4], \quad I_1:=[3/4,7/8].$$
We also fix affine maps:
$$s_{ij}:= s_i\circ s_j:[0,1]\to I_{ij},\quad\text{ where }I_{ij}=s_{ij}([0,1]).$$
Note that $\{I_{ij}\}_{i,j=0,1}$ is a pairwise disjoint family of four intervals in $\II$.
For each $P\in\Omega$ we define
$$P_i:= s_i^{-1}(P\cap I_i)\qquad 0\leq i\leq 1.$$ 
We extend this definition to:
$$P_{ij}=(P_i)_j=s_j^{-1}(s_i^{-1}(P\cap I_i)\cap I_j)=s_{ij}^{-1}(P\cap I_{ij}).$$

We shall define a map from $\Omega$ into the unit ball of the real Hilbert space $\HH=l^2(\N)$. 
We denote the unit ball as: $$B:=\{v\in\HH\;:\; \|v\|\le 1\}.$$
Let us first recall the following well known result.
\begin{thm}[\cite{BS1983}]\label{thm:BS1983}
There exists a Lipschitz map $f:B\to B$ such that: 
$$\inf_{x\in B}\|f(x)-x\|=\delta>0$$
and so that the Lipschitz constant equals $L<\sqrt{2}$.
%Moreover, by choosing a sufficiently small $t>0$ and replacing $f$ by $f_t:=(1-t)\Id+t f$, we obtain such a map with a Lipschitz constant $L<\sqrt{2}$.
\end{thm}
%By choosing a sufficiently small $t>0$ and defining $f_t:=(1-t)\Id+t f$, we obtain a map $f:B\to B$ with a Lipschitz constant $L$ less than $\sqrt2$ such that
%$$\inf_{x\in B}\|f(x)-x\|>0.$$

%Let us give the lexicographical order $<$ on the set of length-two binary words $$\CC_2:=\{00,01,10,11\}.$$
%We extend this order to a partial \emph{anti-chain} order $<$
%$$\CC:=\{0,1,00,01,10,11\}$$ so that $u<v$ in $\CC$
%if and only if $u<v$ in the lexicographic order, 

Let us define the \emph{anti-chain} order $<$ on the set:
$$\CC:=\{0,1,00,01,10,11\}$$ so that $u<v$ in $\CC$
if and only if neither is a prefix of the other, and in the first digit where they differ, $u$ contains a $0$ and $v$ contains a $1$. For instance, we have $0<10$, but $1$ and $11$ are not comparable. 
We note the following.
%\sam{added}
\begin{lem}\label{lem:F-invariant}
\begin{enumerate}
    \item If $a,b\in\CC$ and $g\in F$ satisfy
    $$gs_a=s_b,$$ then we have 
    $(gP)_b=P_a$ for each $P\in \Omega$.
    \item For all $u<v$ in $\CC$, there exists some $g\in F$
    such that 
    $$gs_u=s_0\quad\text{ and }\quad gs_v=s_1.$$
\end{enumerate}\end{lem}
\begin{proof}
(1) We note the computation
\begin{align*}
(gP)_b&=s_b^{-1}\left(gP\cap I_b\right)
=s_a^{-1}g^{-1}\left(gP\cap I_b\right)\\
&=s_a^{-1}\left(P\cap g^{-1}I_b\right)
=s_a^{-1}\left(P\cap I_a\right)=P_a.
\end{align*}

(2) Note that the internal regular dyadic intervals $I_u$ and $I_v$ satisfy
$$\max I_u<\min I_v.$$
By \cref{lem:transitive}, we can choose some $g\in F$ such that $g$ restricts to affine maps $I_u\to I_0$ and $I_v\to I_1$. The maps $gs_u$ and $s_0$ coincide, since both are affine maps from $[0,1]$ to $I_0$. Similarly, we have $gs_v=s_1$, completing the proof.
\end{proof}

We now define a map
recursively as: $c:\Omega\to B$ by 
$$c(P):=\begin{cases}
f\left( \frac{c(P_{0})+c(P_{1})}{2}\right),& \text{ if }P\cap I_0\neq \emptyset, P\cap I_1\neq \emptyset,\\
0,&\text{ otherwise}.
\end{cases}$$
The above map $c$ is well-defined by induction, 
as $|P| \ge  \left|P_i\right|+1$ whenever $P\cap I_{j}\neq \emptyset$ for $0\leq j\leq 1$. 
%\sam{capitalize $Z_i$ for consistency}
For $i,j\in\{0,1\}$, define
$$
Y_i(P):=c(P_i),
\qquad
Y_{ij}(P):=c(P_{ij}),
\qquad
Z_i(P):=\frac{Y_{i0}(P)+Y_{i1}(P)}{2}.
$$

%\sam{changed from $\Omega_2$ to $\Omega^*$, to avoid confusion from $\Omega_A$}
%\sam{changed from ``. and observe'' to ``, and observe''}
We also define
$$
\Omega^*:=
\left\{
P\in\Omega:
P_{ij}\neq\varnothing
\text{ for all }i,j\in\{0,1\}
\right\},
$$
and observe:

\begin{lem}
$\mu(\Omega^*)=1$.
\end{lem}

\begin{proof}
Consider any finite set $A\subset D$ such that $A\cap I_i\neq \emptyset$ and $A\cap I_{ij}\neq \emptyset$ for each $i,j\in\{0,1\}$.
Then $\mu(\Omega_A)=1$ and $\Omega_A\subseteq\Omega^*$.
\end{proof}

% Requires amsmath, amssymb, and amsthm.
% In the preamble, if not already defined:
% \newtheorem{lemma}{Lemma}
The following Lemma shows certain crucial identities that hold for $P\in \Omega^*$.

%Note that all the identities in part (3) of the previous Lemma hold on a set of
%$\mu$-measure one, for instance $\Omega^*$.
%We are hence able to state and prove the following.

\begin{lem}\label{lem:final}
Assume that $f$ is as in Theorem \ref{thm:BS1983} and $L$-Lipschitz.
%$$
%\delta:=\inf_{v\in B}\|f(v)-v\|>0,
%$$
Then, for every $P\in\Omega^*$, the following holds.

\begin{enumerate}
\item $Y_i(P)=f(Z_i(P))$ for every $i\in \{0,1\}$.
\item $
\|Y_0(P)-Y_1(P)\|
\le L\|Z_0(P)-Z_1(P)\|
$.

\item $
\|Y_i(P)-Z_i(P)\|\ge\delta$ for each $i\in \{0,1\}$.

\end{enumerate}

\end{lem}

\begin{proof}
Let $P\in\Omega^*$. The fact that
$Y_i(P)=f(Z_i(P))$ follows from the definition since $P\in \Omega^*$ implies that $P_0,P_1\neq \emptyset$.
Using this, the Lipschitz property gives us
$$
\|Y_0(P)-Y_1(P)\|
=
\|f(Z_0(P))-f(Z_1(P))\|
\le L\|Z_0(P)-Z_1(P)\|.
$$
Also, $Z_i(P)\in B$, so the definition of $\delta$
gives
$$
\|Y_i(P)-Z_i(P)\|
=
\|f(Z_i(P))-Z_i(P)\|
\ge\delta.\qedhere
$$
\end{proof}

%\sam{added a paragraph for the idea}
We will now establish that the vectors $Y_i(P)$ and $Y_{ij}(P)$ are at the equal distances from each other on average, over $P\in\Omega^*$. It will then follow that the distances between $Z_0$ and $Z_1$ are uniformly shorter, 
which by the Lipschitz property of $f$ implies that these vectors all coincide in the sense of mean-distance. 
This would induce a contradiction upon applying the uniform lower bound for the displacement of $f$. 

To be more concrete, recall that $\mu$ defines a finitely additive left-invariant integral $\mathbb{E}$, which is a positive, normalized, linear functional:
$$\bE:L^\infty(\Omega;\bR)\to \bR\qquad \bE(\phi\circ g)=\bE(\phi), \forall \phi\in L^\infty(\Omega;\bR), \forall g\in F$$ 
defined as:
$$\bE(\phi):=\int_{\Omega}\phi d\mu$$
%\sam{lemma $=>$ Lemma}
The following lemma is where the invariance of our mean is used.

\begin{lem}\label{lem:invariance2}
For all $u,v$ in $\CC=\{0,1,00,01,10,11\}$ satisfying the anti-chain order $u<v$, we have that
$$ \bE\abss{Y_u-Y_v}^2 = \bE\abss{Y_0-Y_1}^2.$$
\end{lem}
\begin{proof}
By \cref{lem:F-invariant}, there exists some $g\in F$ such that
$gs_u=s_0$ and $gs_v=s_1$. Using part (1) of the same lemma
and the $F$-invariance of $\mu$, we see that
\begin{align*}
\bE\abss{Y_u-Y_v}^2
&= \int_\Omega \abss*{Y_u(P)-Y_v(P)}^2 d\mu(P)
= \int_\Omega \abss*{c(P_u)-c(P_v)}^2 d\mu(P)\\
&
= \int_\Omega \abss*{c((gP)_0)-c((gP)_1)}^2 d\mu(P)=\bE\abss{Y_0-Y_1}^2.\qedhere
\end{align*}
%The proof is complete.
\end{proof}
 
%\begin{lem}\label{lem:invariance2} For all $u,v$ in $\CC=\{0,1,00,01,10,11\}$ satisfying the anti-chain order $u<v$, we have that $$ \bE\abss{Y_u-Y_v} = \bE\abss{Y_0-Y_1}.$$ The following holds: \begin{enumerate} \item For each pair $(i,j)\in \{0,1\}^2, k\in \{0,1\}, i\neq k$, it holds that: $$\mathbb{E}||Y_{ij}-Y_{k}||^2=\mathbb{E}||Y_0-Y_1||^2$$ \item For each $(i,j)\neq (k,l)\in \{0,1\}^2$, $$\mathbb{E}||Y_{ij}-Y_{kl}||^2=\mathbb{E}||Y_0-Y_1||^2$$ \end{enumerate} \end{lem} \begin{proof} We prove part $(1)$, and the proof of part $(2)$ is similar. We assume without loss of generality that $\sup(I_{ij})<\inf(I_k)$. The proof for the case $\sup(I_{k})<\inf(I_{ij})$ is similar. The proof follows from the last part of Lemma \ref{lem:transitive} which provides an element $g\in F$ such that: \begin{enumerate} \item $g(I_{ij})=I_0$ and $g(I_{k})=I_1$. \item The restrictions of $g$ to $I_{ij},I_{k}$ are affine. \item $\mathbb{E}$ is $F$-invariant. \end{enumerate} \end{proof}
 
Recall the following classical result of Apollonius on Euclidean geometry, which also works in Hilbert spaces.
\begin{prop}\label{prop:Ap}
Given vectors $x,a,b\in \HH$, the following holds.
$$
\abss*{ x-\frac{a+b}2}^2=\frac12\left(
\|x-a\|^2+\|x-b\|^2-\frac{\|a-b\|^2}2
\right).
$$
\end{prop}

\begin{proof}[Proof of Theorem \ref{thm:main}]
First, we fix $q$ to denote the quantity $\bE \|Y_0-Y_1\|^2$.
%\sam{Note that Lemma tells us that $=>$ Note from Lemma that}
Note from Lemma \ref{lem:invariance2} that the integrals described in the statement are also equal to $q$.
We will be making repeated use of this fact in what follows below.
%\sam{Style change $\| \cdot \|$ $=>$ abss delimiter}
Applying Apollonius Theorem (\cref{prop:Ap}), 
to $Z_0$ and $Z_1= (Y_{10}+Y_{11})/2$, we obtain
\begin{align*}
\abss*{Z_0-Z_1}^2
&
%=\abss*{Z_0-\frac{Y_{10}+Y_{11}}{2}}^2
=\frac12
\left(
\abss*{Z_0-Y_{10}}^2+\abss*{Z_0-Y_{11}}^2-\frac{\abss*{Y_{10}-Y_{11}}^2}2
\right)
\end{align*}
%\sam{deductions broken into two parts, for a later reference by \cref{thm:main2}}
Applying the theorem again to $Z_0=(Y_{00}+Y_{01})/2$, we have
\begin{equation}\label{eq:z0z1}
\abss*{Z_0-Z_1}^2
=\frac14
\left(\sum_{i,j=0,1} \abss*{Y_{0i}-Y_{1j}}^2-
\sum_{i=0,1} \abss*{Y_{i0}-Y_{i1}}^2
\right).
\end{equation}
So we obtain, upon applying $\bE$ to both sides over $\Omega^*$ (which has measure $1$), and applying Lemma \ref{lem:invariance2}, that $\bE \|Z_0-Z_1\|^2=q/2$.

%\sam{Claim environment used}
\begin{claim*}
We have $q=0$.
\end{claim*}

Assume by way of contradiction that $q>0$. Then since $L<\sqrt{2}$ and $\|Y_0-Y_1\|^2\le L^2 \|Z_0-Z_1\|^2$ for the full measure set $\Omega^*$, the Lipschitz property implies
$$q=\bE \|Y_0-Y_1\|^2\le L^2 \bE \|Z_0-Z_1\|^2=\frac{qL^2}{2}< q$$
This is a contradiction, and so we obtain that $q=0$. 
%Note that here we used the identities from \cref{lem:invariance2} that hold
%for the full measure set $\Omega^*$.

%\sam{Fixed an error: now use the distance between $Y_1$ and $Y_{00}$, for a specific choice of generators} 
By Apollonius's Theorem and \cref{lem:invariance2} again, we have:
\begin{equation}\label{eq:y0z1}
\|Y_1-Z_0\|^2
\le 2\abss*{Y_1-Y_{00}}^2+2\abss*{Y_{00}-Z_0}^2
=2\abss*{Y_1-Y_{00}}^2+ \frac{\abss{Y_{00}-Y_{01}}^2}2.
%2\|Y_1-Z_0\|^2&\le  \|Y_1-Y_{00}\|^2+\|Y_{0}-Y_{11}\|^2-\frac{\|Y_{10}-Y_{11}\|^2}2.
\end{equation}
By Lemma~\ref{lem:invariance2}, we have $\bE\|Y_1-Z_0\|^2\le 2q+ q/2 = 0$. 
It follows that
$$\bE\|Y_1-Z_1\|^2=\bE\|Y_1-Z_0+Z_0-Z_1\|^2
\le 2\bE\|Y_1-Z_0\|^2+2\bE \|Z_0-Z_1\|^2=0.
$$
This contradicts the hypothesis that $\inf_{x\in B}\|f(x)-x\|>0$.
This means that our assumption that the $F$-action on $\Omega$ is extensively amenable is false. 
Therefore, $F$ is nonamenable.
The proof is complete.
\end{proof}

The above proof uses the properties of $F$ rather loosely; namely, it only uses the existence of certain generators in $F$ that are ``equivariantly'' acting on $[0,1]$ with respect to the affine contractions $s_i$. 
%\sam{Explained the efficient choice of generators}
More precisely, the integral $\bE$ was only applied to Equations \eqref{eq:z0z1} and \eqref{eq:y0z1}.
Setting $$\CC_2:=\{00,01,10,11\},$$
 we used that
$$\bE\abss{Y_u-Y_v}^2=\bE\abss{Y_0-Y_1}^2
=\bE\abss{Y_{00}-Y_1}^2$$
for $u<v$ in $\CC_2$.
%\sam{changed $a$ to $h_0$}
Hence, the argument naturally generalizes to the setting of abstract sets with an efficient choice of generators $h_0$ and $h_{uv}$ below. We omit the proof, which is a straightforward analogue of that of \cref{thm:main}. We denote by $\Sym(X)$ the permutation group on a set $X$.

\begin{thm}\label{thm:main2}
Let $X$ be a set, and let $A_0,A_1$ be disjoint nonempty subsets
equipped with bijections $s_i:X\to A_i$ for $i=0,1$.
Setting $s_{ij}:=s_i s_j$,
%and $A_{ij}:=s_{ij}(X)$.
we assume the following.
%\sam{changed from $u,v$ to $u<v$ in (i)}
\begin{enumerate}[(i)]
\item For each $u < v\in\CC_2$, there exists some $h_{uv}\in \Sym(X)$ such that $h_{uv}s_u=s_0$, and such that $h_{uv}s_v=s_1$.
\item There exists some $h_0\in \Sym(X)$ such that 
$h_0s_{00}=s_0$ and $h_0s_1=s_1$.
\end{enumerate}
Then the group 
$$G:=\form*{h_0,h_{uv}\in\Sym(X)\;:\; u<v\text{ in }\CC_2}$$
is not amenable.
\end{thm}

\bibliographystyle{plain}
\bibliography{ref}
\end{document}